\documentclass[letterpaper, 10pt, conference]{ieeeconf}
 
\IEEEoverridecommandlockouts                              

\usepackage{amsmath}    
\usepackage{amsfonts}
\usepackage{graphicx}   
\usepackage{subcaption}
\usepackage{setspace}
\usepackage{epsfig} 
\usepackage{color}
\usepackage[normalem]{ulem}
\usepackage{cancel}
\usepackage{amssymb}
\usepackage{color}
\usepackage{float}
\usepackage[ruled,vlined,titlenotnumbered]{algorithm2e} 

\usepackage{todonotes} 

\newcommand{\R}{\mathbb{R}} 
\newcommand{\state}{z} 
\newcommand{\spart}{y} 
\newcommand{\scstate}{x} 

\newcommand{\traj}{\zeta} 

\newcommand{\ctrl}{u} 
\newcommand{\cset}{\mathcal{U}}
\newcommand{\cfset}{\mathbb{U}}

\newcommand{\genset}{\mathcal{S}}
\newcommand{\xset}{\mathcal{X}}
\newcommand{\zset}{\mathcal{Z}}

\newcommand{\proj}{\text{proj}}
\newcommand{\bp}{\proj^{-1}}

\newcommand{\fdyn}{f} 
\newcommand{\scdyn}{g} 
\newcommand{\sctraj}{\xi}

\newcommand{\valfunc}{V} 
\newcommand{\fc}{l} 
\newcommand{\costate}{p}

\newcommand{\ham}{H} 
\newcommand{\brs}{\mathcal{V}} 
\newcommand{\targetset}{\mathcal{L}} 
\newcommand{\pos}{p} 
\newcommand{\vel}{v} 

\newtheorem{alg}{Algorithm}

\newtheorem{defn}{Definition}

\newtheorem{thm}{Theorem}
\newtheorem{lem}{Lemma}
\newtheorem{cor}{Corollary}
\newtheorem{IEEEproof}{Proof}

\title{\LARGE \bf Exact and Efficient Hamilton-Jacobi-based Guaranteed Safety Analysis via System Decomposition}

\author{Mo Chen, Sylvia Herbert, Claire J. Tomlin
\thanks{This work has been supported in part by NSF under CPS:ActionWebs (CNS-931843), by ONR under the HUNT (N0014-08-0696) and SMARTS (N00014-09-1-1051) MURIs and by grant N00014-12-1-0609, by AFOSR under the CHASE MURI (FA9550-10-1-0567). The research of M. Chen has received funding from the ``NSERC PGS-D'' Program.}
\thanks{All authors are with the Department of Electrical Engineering and Computer Sciences, University of California, Berkeley. \{mochen72, sylvia.herbert, tomlin\}@berkeley.edu}
}

\begin{document}
\maketitle
\thispagestyle{empty}
\pagestyle{empty}

\begin{abstract}
Hamilton-Jacobi (HJ) reachability is a method that provides rigorous analyses of the safety properties of dynamical systems. This method has been successfully applied to many low-dimensional dynamical system models such as coarse models of aircraft and quadrotors in order to provide safety guarantees in potentially dangerous scenarios. These guarantees can be provided by the computation of a backward reachable set (BRS), which represents the set of states from which the system may be driven into violating safety properties despite the system's best effort to remain safe. Unfortunately, HJ reachability is not practical for high-dimensional systems because the complexity of the BRS computation scales exponentially with the number of state dimensions. Although numerous approximation techniques are able to tractably provide conservative estimates of the BRS, they often require restrictive assumptions about system dynamics without providing an exact solution. In this paper we propose a general method for decomposing dynamical systems. Even when the resulting subsystems are coupled, relatively high-dimensional BRSs that were previously intractable or expensive to compute can now be quickly and exactly computed in lower-dimensional subspaces. As a result, the curse of dimensionality is alleviated to a large degree without sacrificing optimality. We demonstrate our theoretical results through two numerical examples: a 3D Dubins Car model and a 6D Acrobatic Quadrotor model.
\end{abstract}

\section{Introduction}
As the presence of safety-critical systems in everyday life has grown, so has the importance for the verification of these systems. Within the next decade we expect to see a rapid increase in the use of safety-critical systems such as autonomous cars, unmanned aerial vehicles, and other robots. Given the number and density of autonomous systems expected in civilian space, higher-fidelity models are needed to more accurately characterize these systems so that safety can be guaranteed. In addition, analysis of higher-dimensional dynamical system models has the potential to provide valuable insight into the behavior of system states that are frequently ignored to keep the system dimensionality low. Thus, tractable verification tools that are not overly conservative are urgently needed.

Optimal control and differential game theory are powerful tools for the verification of non-linear systems due to their flexibility with respect to system dynamics, treatment of unknown disturbances, and guaranteed optimality \cite{Barron90, Mitchell05, Margellos11, Bokanowski11}. Reachability analysis is core to these methods; here, the goal is to compute the backward reachable set (BRS), defined as the set of states from which the system can be driven into some unsafe set despite using the optimal control to avoid the unsafe set. Hamilton-Jacobi (HJ) reachability has been successfully used to guarantee safety for low-dimensional systems in application such as pair-wise collision avoidance \cite{Mitchell05}, automated aerial refueling \cite{Ding08}, and many others \cite{Vaisbord88, Mitchell02}. HJ reachability theory is also very convenient to use due to the many numerical tools available to obtain optimal solutions \cite{Sethian96, Osher02, Mitchell09}.

 Despite these advantages, HJ reachability can be impractical for many high-dimensional systems due to issues with scaling. HJ reachability-based methods involve solving a partial differential equation (PDE) or variational inequality on a grid representing a numerical discretization of the state space. As a result, the computation complexity scales exponentially with the system dimension. Application of current formulations of HJ reachability is limited to systems with approximately five dimensions or fewer, making the verification of most high-dimensional system models intractable.

For the analysis of high-dimensional systems, a number of approximation techniques exist. Unfortunately, these techniques usually place strong assumptions on system dynamics, such as requiring a polynomial form \cite{Majumdar13, Dreossi16}, a linear form \cite{Kurzhanski00, Maidens13}, or a Hamiltonian that is only dependent on the control variable \cite{Darbon16}. Other methods that are less restrictive in terms of system dynamics include \cite{Mitchell03}, which works with projections, and \cite{Chen16b}, which involves treating system states as disturbances. In all of the methods mentioned so far, varying degrees of approximation or conservatism is introduced. Under some special scenarios such as those outlined in \cite{Mitchell11} or \cite{Fisac15}, a small dimensionality reduction may be possible when obtaining exact optimal solutions.

The previous methods either are forced to trade off between optimality and computation complexity or provide only a small dimensionality reduction. In contrast, this paper presents the \textit{self-contained subsystem (SCS)} formulation for computing \textit{exact, optimal solutions} of systems with dynamics while drastically reducing dimensionality. Motivated by the need to provide safety guarantees, we compute BRSs in lower-dimensional subspaces of the full system state space, and then combine these low-dimensional BRSs to exactly construct the full-dimensional BRS. The full-dimensional BRS can be exactly constructed through back projections of the lower-dimensional BRSs \textit{even with coupling between the different subsystems}. Furthermore, the theory we present in this paper is compatible with any other method such as \cite{Chen16b} and \cite{Mitchell11}. When different methods are combined together, even more substantial dimensionality reduction can be achieved.

This paper will be presented as follows:
\begin{itemize}
\item First, in Sections \ref{sec:background} and \ref{sec:formulation} we introduce the HJ reachability theory relevant to our paper, and all the definitions needed for our proposed HJ-based system decomposition.
\item Next, in Section \ref{sec:sc} we present the SCS formulation, our main theoretical result. We describe how BRSs in lower-dimensional subspaces can be combined to construct the full-dimensional BRS exactly.
\item Finally, in Section \ref{sec:hd} we present two numerical examples: a low-dimensional 3D Dubins Car example to validate our theory and a high-dimensional 6D Acrobatic Quadrotor example that was previously intractable using standard methods.
\end{itemize}

\section{Background \label{sec:background}}
There are several HJ formulations that can compute BRSs exactly when the system dimensionality is low. Although these methods have been successfully used for lower-dimensional systems, they become intractable when the system dimension is greater than approximately five. In this section, we give a brief overview to provide a starting point on which we build the new proposed theory.

\subsection{Full System Dynamics}
\begin{defn}
\textbf{Full system}. Let $\state$ be the state variable of the system under consideration. We call this system the ``full system," or just ``system" for short. The evolution of the state of the full system satisfies the ordinary differential equation (ODE)
\begin{equation}
\begin{aligned}
\label{eq:fdyn}
\frac{d\state}{ds} = \dot\state = \fdyn(\state, \ctrl), s \in [t, 0] \\
\state \in \zset, \ctrl \in \cset
\end{aligned}
\end{equation}
\end{defn}

For clarity, we assume that the state space $\zset$ is $\R^n$, but our theory also applies to systems with periodic state dimensions such as angles. The control is denoted by $\ctrl$, with the control function $\ctrl(\cdot)$ being drawn from the set of measurable functions\footnote{A function $f:X\to Y$ between two measurable spaces $(X,\Sigma_X)$ and $(Y,\Sigma_Y)$ is said to be measurable if the preimage of a measurable set in $Y$ is a measurable set in $X$, that is: $\forall V\in\Sigma_Y, f^{-1}(V)\in\Sigma_X$, with $\Sigma_X,\Sigma_Y$ $\sigma$-algebras on $X$,$Y$.}:
\vspace{-0.1in}
\begin{equation}
\ctrl(\cdot) \in \cfset(t) = \{\phi: [t, 0] \rightarrow \cset: \phi(\cdot) \text{ is measurable}\}
\end{equation}

The system dynamics $\fdyn: \zset \times \cset \rightarrow \zset$ is assumed to be uniformly continuous, bounded, and Lipschitz continuous in $\state$ for fixed $\ctrl$. With this assumption, given $\ctrl(\cdot) \in \cfset$, there exists a unique trajectory solving \eqref{eq:fdyn} \cite{Evans84, Coddington55}. 

We will denote solutions, or trajectories of \eqref{eq:fdyn} starting from some state $\state$ at time $t$ under control $\ctrl(\cdot)$ as $\traj(s; \state, t, \ctrl(\cdot))$. The system trajectory satisfies an initial condition and the ODE \eqref{eq:fdyn} almost everywhere:
\begin{equation}
\label{eq:fdyn_traj}
\begin{aligned}
\frac{d}{ds}\traj(s; \state, t, \ctrl(\cdot)) &= \fdyn(\traj(s; \state, t, \ctrl(\cdot)), \ctrl(s)) \\
\traj(t; \state, t, \ctrl(\cdot)) &= \state
\end{aligned}
\end{equation}

\subsection{Backward Reachable Set}
\label{sec:RSRT}
In this paper, we consider a common definition of the BRS relevant for guaranteeing safety. Intuitively, the BRS represents the set of states $\state$ from which the system can be driven into an unsafe set $\targetset$ at a particular time. For our definition of BRS, we stipulate that the system be driven to $\targetset$ for all control functions $\ctrl(\cdot)$. In this case, the unsafe set can often be interpreted as a set of states to be avoided (such as an obstacle), and the BRS represents the set of states that leads to the system entering the unsafe set despite all possible control functions. We now formally define the BRS.

\begin{defn}
\label{defn:brs}
\textbf{Backward reachable set}. We denote the BRS $\brs(t)$, and define it as follows:
\begin{equation}
\label{eq:rset_avoid}
\brs(t) = \{\state \in \zset: \forall \ctrl(\cdot) \in \cfset, \traj(0; \state, t, \ctrl(\cdot)) \in \targetset \}
\end{equation}
\end{defn}

\subsection{The Full Formulation for Computing the BRS}
There are various similar HJ formulations such as \cite{Barron90, Mitchell05, Bokanowski11}, and \cite{Varaiya67} that cast the reachability problem as an optimal control problem and directly compute the BRS in the full state space of the system. These numerical solutions to the optimal control problem involve solving an HJ PDE on a grid that represents a discretization of the state space. Although these methods are not scalable beyond relatively low-dimensional systems, they form the foundation on which we will build our theory. We now briefly summarize the necessary details related to the HJ PDEs, and what their solutions represent in terms of the cost function and value function of the corresponding optimal control problem.

Let the unsafe set $\targetset \subseteq \zset$ be represented by the implicit surface function $\fc(\state)$ such that the unsafe set is the zero sub-level set of the implicit surface function: $\targetset = \{\state \in \zset: \fc(\state) \le 0\}$. Such a function always exists since we can choose $\fc(\cdot)$ to be the signed distance function from $\targetset$. Examples of implicit surface functions are shown as colored surfaces in Fig. \ref{fig:implicit_sf}, with the boundary of the corresponding sets they represent shown in black.

\begin{figure}
	\centering
	\includegraphics[width=0.47\textwidth]{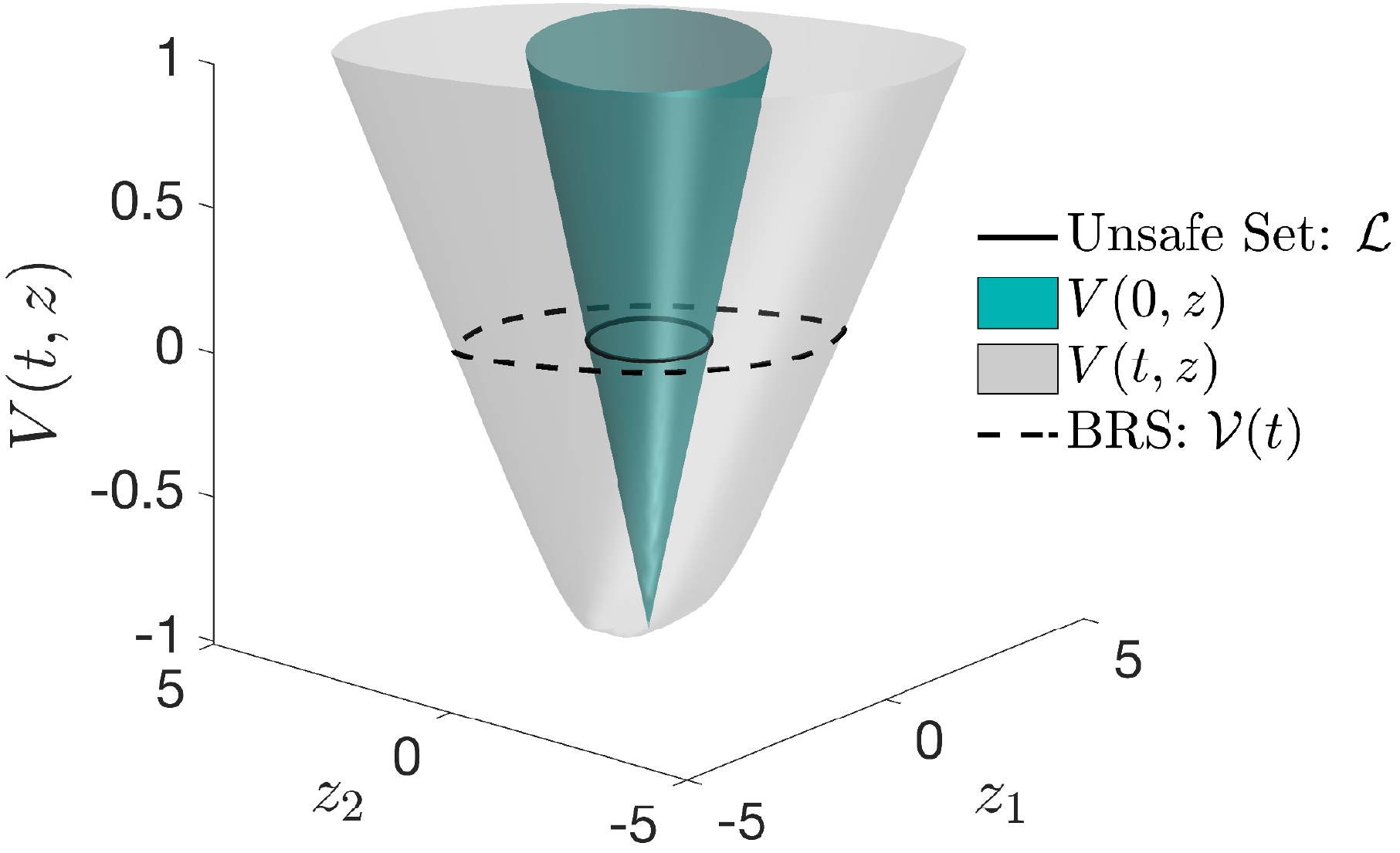}
	\caption{A simple 2D example illustrating HJ reachability. The boundary of the unsafe set $\targetset$ in the state space is shown as the solid black line. The blue surface represents the implicit surface function $\fc(\state)$ of the unsafe set, which by \eqref{eq:set_hj} is equivalent to $\valfunc(0, \state)$. The light gray surface shows the value function at some $t<0$: $\valfunc(t, \state)$. The corresponding BRS $\brs(t)$ is the zero sub-level set of this function; the boundary of $\brs(t)$ is seen here as the dashed black line. If the system remains outside of the BRS at $t<0$, it is guaranteed to not enter the unsafe set at $t=0$.}
	\label{fig:implicit_sf}
\end{figure}

Consider the optimal control problem given by

\begin{equation}
\label{eq:oc}
\begin{aligned}
&\valfunc(t, \state) = \max_{\ctrl(\cdot)\in\cfset} \fc(\traj(0; \state, t, \ctrl(\cdot))) \\
&\qquad \text{subject to \eqref{eq:fdyn_traj}}
\end{aligned}
\end{equation}

\noindent with the optimal control being given by
\vspace{-0.1in}
\begin{equation}
\label{eq:aoc}
\ctrl^*(\cdot) = \arg \max_{\ctrl(\cdot)\in\cfset} \fc(\traj(0; \state, t, \ctrl(\cdot)))
\end{equation}

It is well-known that the value function $\valfunc(t, \state)$ is the implicit surface function representing $\brs(t)$: $\brs(t) = \{\state \in \zset: \valfunc(t, \state) \le 0 \}$.

The value function $\valfunc(t,\state)$ is the viscosity solution \cite{Crandall83, Crandall84} of the HJ PDE
\vspace{-0.05in}
\begin{equation}
\label{eq:set_hj}
\begin{aligned}
D_s \valfunc(s, \state) + \ham(\state, \nabla\valfunc(s, \state)) &= 0, \quad s \in [t, 0] \\
\valfunc(0, \state) &= \fc(\state)
\end{aligned}
\end{equation}

The Hamiltonian in \eqref{eq:set_hj} is given by

\begin{equation}
\label{eq:ham}
\ham(\state, \costate) = \max_{\ctrl\in\cset}\costate \cdot \fdyn(\state, \ctrl)
\end{equation}

Fig. \ref{fig:implicit_sf} shows an illustration of HJ reachability. $\fc(\state)$, the implicit surface function representing $\targetset$, and the value function $\valfunc(t,\state)$, the implicit surface function representing the BRS $\brs(t)$, are shown as the blue and light gray surfaces respectively. The unsafe set $\targetset$ and the BRS $\brs(t)$ are the  zero sub-level sets of these two surface functions; the boundaries of $\targetset$ and $\brs(t)$ are shown in black. Once the value function $\valfunc$ is computed, the optimal control \eqref{eq:aoc} can be obtained by the expression

\begin{equation}
\ctrl^*(s) = \arg \max_{\ctrl\in\cset} \nabla \valfunc(s, \state) \cdot \fdyn(\state, \ctrl)
\end{equation}

We state the following algorithm for clarity and convenience:

\begin{alg}
\textbf{Full formulation}. Given an unsafe set $\targetset$ and dynamics \eqref{eq:fdyn}, the full formulation for computing the BRS is given by the following algorithm:
\begin{enumerate}
\item Define the implicit surface function $\fc(\state)$.
\item Solve \eqref{eq:set_hj} with Hamiltonian \eqref{eq:ham} to obtain $\valfunc(t, \state)$, the implicit surface function representing $\brs(t)$.
\end{enumerate}
\end{alg}

\section{Problem Formulation \label{sec:formulation}}
In this paper, we seek to obtain the BRS in Definition \ref{defn:brs} via computations in a lower-dimensional subspace under the assumption that the system \eqref{eq:fdyn} can be decomposed into SCSs. Such a decomposition can be commonly found, since many systems involve components that are loosely coupled. In particular, in the dynamics of many vehicles, the evolution of the position variables is often weakly coupled though other variables such as heading.

We now proceed with some essential definitions required to precisely state our main results.
\vspace{-0.05in}
\subsection{Definitions}
\subsubsection{\textbf{Subsystem Dynamics}}
Let the system $\state \in \zset = \R^n$ be partitioned as follows:
\vspace{-0.1in}
\begin{equation}
\begin{aligned}
&\state = (\spart_1, \spart_2, \spart_3) \\
&\spart_1 \in \R^{n_1}, \spart_2 \in \R^{n_2}, \spart_3 \in \R^{n_3} \\
&n_1, n_2 > 0, n_3 \ge 0 \\
\end{aligned}
\end{equation}

Note that $n_3$ could be zero, and $n_1 + n_2 + n_3 = n$. We call the variables $\spart_i$ the ``state partitions'', or just ``partitions'', of the system.

Define the SCS states $\scstate_1 \in \xset_1 = \R^{n_1+n_3}, \scstate_2 \in \xset_2 = \R^{n_2+n_3}$ as follows:
\begin{equation}
\label{eq:scstates}
\begin{aligned}
\scstate_1 &= (\spart_1, \spart_3) \\
\scstate_2 &= (\spart_2, \spart_3) \\
\end{aligned}
\end{equation}

It is important to note that $\scstate_1$ and $\scstate_2$ in general have overlapping states in the partition $\spart_3$. Note that our theory is applicable to any finite number of subsystems defined in the analogous way; however, for clarity and without loss of generality, in this paper we will assume that there are two subsystems.

For convenience, we have assumed that $\xset_1 = \R^{n_1+n_3}, \xset_2 = \R^{n_2+n_3}$, but as previously mentioned, our theory also applies to systems with periodic state dimensions.

\begin{defn}
\textbf{Self-contained subsystem}. We call each of the systems with states $\scstate_i$ evolving according to \eqref{eq:scdyn} a ``self-contained subsystem'' (SCS), or just "subsystem" for short.
\vspace{-0.1in}
\begin{equation}
\label{eq:scdyn}
\begin{aligned}
\frac{d\scstate_1}{ds} = \dot\scstate_1 &= \scdyn_1(\scstate_1, \ctrl_1) = \scdyn_1(\spart_1, \spart_3, \ctrl_1), \quad s \in [t, 0] \\
\frac{d\scstate_2}{ds} = \dot\scstate_2 &= \scdyn_2(\scstate_2, \ctrl_2) = \scdyn_1(\spart_2, \spart_3, \ctrl_2) \\
\ctrl_1 \in \cset_1, &\ctrl_2 \in \cset_2
\end{aligned}
\end{equation}
\end{defn}

Intuitively \eqref{eq:scdyn} means that the evolution of states in each subsystem depend only on the states in that subsystem: for example, the evolution of $\scstate_1$ depends only on the states in $\scstate_1$. However, the two subsystems are coupled through the state partition $\spart_3$. Note that the subsystem controls $\ctrl_1$ and $\ctrl_2$ depend on how the control inputs appear in subsystem states $\scstate_1$ and $\scstate_2$, and may not exist in some subsystems. For example, consider the dynamics of a Dubins Car:

\begin{equation}
\label{eq:dubins_car}
\begin{aligned}
\left[ \begin{array}{c}
\dot\pos_x\\
\dot\pos_y \\
\dot\theta
\end{array} \right]
=
\left[
\begin{array}{c}
v \cos\theta \\
v \sin\theta \\
\omega
\end{array}\right]
\end{aligned}
\end{equation}

\noindent with state $\state = (\pos_x, \pos_y, \theta)$ and control $\ctrl = \omega$. The state partitions are $\spart_1 = \pos_x, \spart_2 = \pos_y, \spart_3 = \theta$. The subsystems $\scstate_i$ and the subsystem controls $\ctrl_i$ are
\begin{equation}
\label{eq:dubins_car_decomp}
\begin{aligned}
\dot{\scstate_1} = 
\left[ \begin{array}{c}
\dot{\spart_1}\\
\dot{\spart_3}
\end{array} \right]
=
\left[ \begin{array}{c}
\dot{\pos_x}\\
\dot{\theta}
\end{array} \right]
&=
\left[\begin{array}{c}
v \cos\theta\\
\omega \\
\end{array}\right]\\
\dot{\scstate_2} = 
\left[ \begin{array}{c}
\dot{\spart_2}\\
\dot{\spart_3}
\end{array} \right]
=
\left[ \begin{array}{c}
\dot{\pos_y} \\
\dot{\theta}
\end{array} \right]
&=
\left[\begin{array}{c}
v \sin\theta \\
\omega
\end{array}\right]\\
\ctrl_1 = \ctrl_2 = \omega &= \ctrl
\end{aligned}
\end{equation}

\noindent where the overlapping state is $\theta = \spart_3$.

The subsystem control signal spaces $\cset_1,\cset_2$ and control function spaces $\cfset_1, \cfset_2$ are defined appropriately according to the full system control signal and function spaces $\cset$ and $\cfset$ based on how the control enters the dynamics of the subsystems. For another example of a system decomposed into two self-contained subsystems, see \eqref{eq:quad6D} and \eqref{eq:quad6Dsc}.

Although there may be common or overlapping states in $\scstate_1$ and $\scstate_2$, the evolution of each subsystem does not depend on the other explicitly. In fact, if we for example entirely ignore the subsystem $\scstate_2$, the evolution of the subsystem $\scstate_1$ is well-defined and can be considered a full system on its own; hence, each subsystem is self-contained.

\subsubsection{\textbf{Projection Operators}}
Define the projection of a state $\state$ onto a subsystem state space $\xset_i$ as
\begin{equation}
\label{eq:project_pt}
\proj_{\xset_i}(\state) = \scstate_i, i = 1,2
\end{equation}

For convenience, we will define the projection operator applied on sets $\genset \subseteq \zset$:
\begin{equation}
\begin{aligned}
\label{eq:project_set}
\proj_{\xset_i}(\genset) &= \{\scstate_i \in \xset_i: \exists \state \in \genset, \proj_{\xset_i}(\state) = \scstate_i\} \\
\end{aligned}
\end{equation}

Since we will aim to relate the BRSs of the subsystems to the BRS of the full system, we also define the back projection operator as
\vspace{-0.1in}
\begin{equation}
\label{eq:backproject_pt}
\bp(\scstate_i) = \{\state \in \zset: \proj_{\xset_i}(\state) = \scstate_i\}
\end{equation}

We will also apply the back projection operator on sets. In this case, we abuse notation and define the back projection operator on some set $\genset_i \subseteq \xset_i$ as
\begin{equation}
\label{eq:backproject_set}
\bp(\genset_i) = \{\state \in \zset: \exists \scstate_i \in \genset_i, \proj_{\xset_i}(\state) = \scstate_i\}
\end{equation}

Fig. \ref{fig:proj_pt} and \ref{fig:proj_set} illustrate the definitions involving projections.

\begin{figure}[H]
	\centering
	\includegraphics[width=0.35\textwidth]{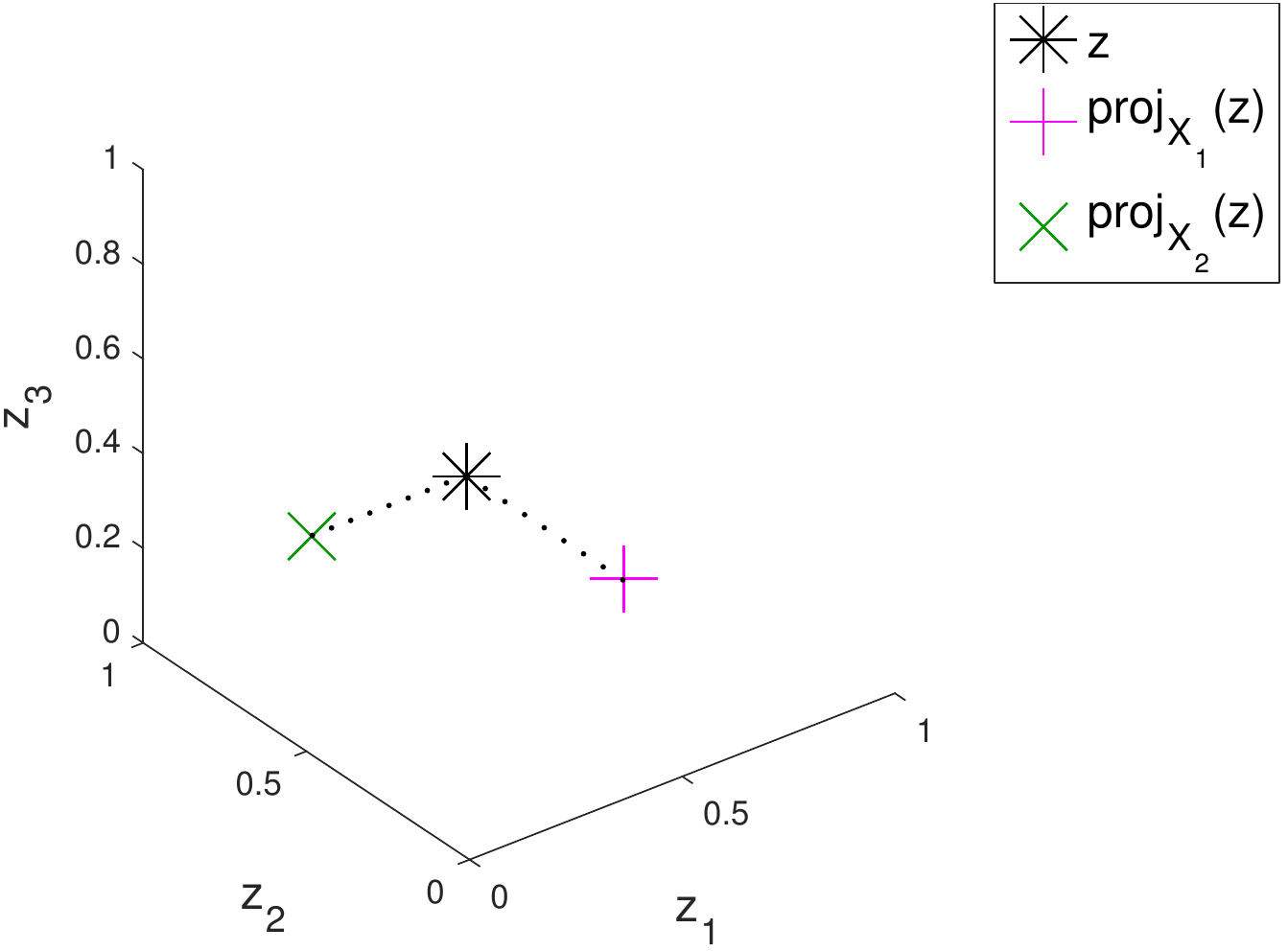}
	\caption{Projection of a point $\state$ onto the lower-dimensional subspaces in the $\state_2$-$\state_3$ plane and the $\state_1$-$\state_3$ plane.}
	\label{fig:proj_pt}
\end{figure}

\begin{figure}[H]
	\centering
	\includegraphics[width=0.5\textwidth]{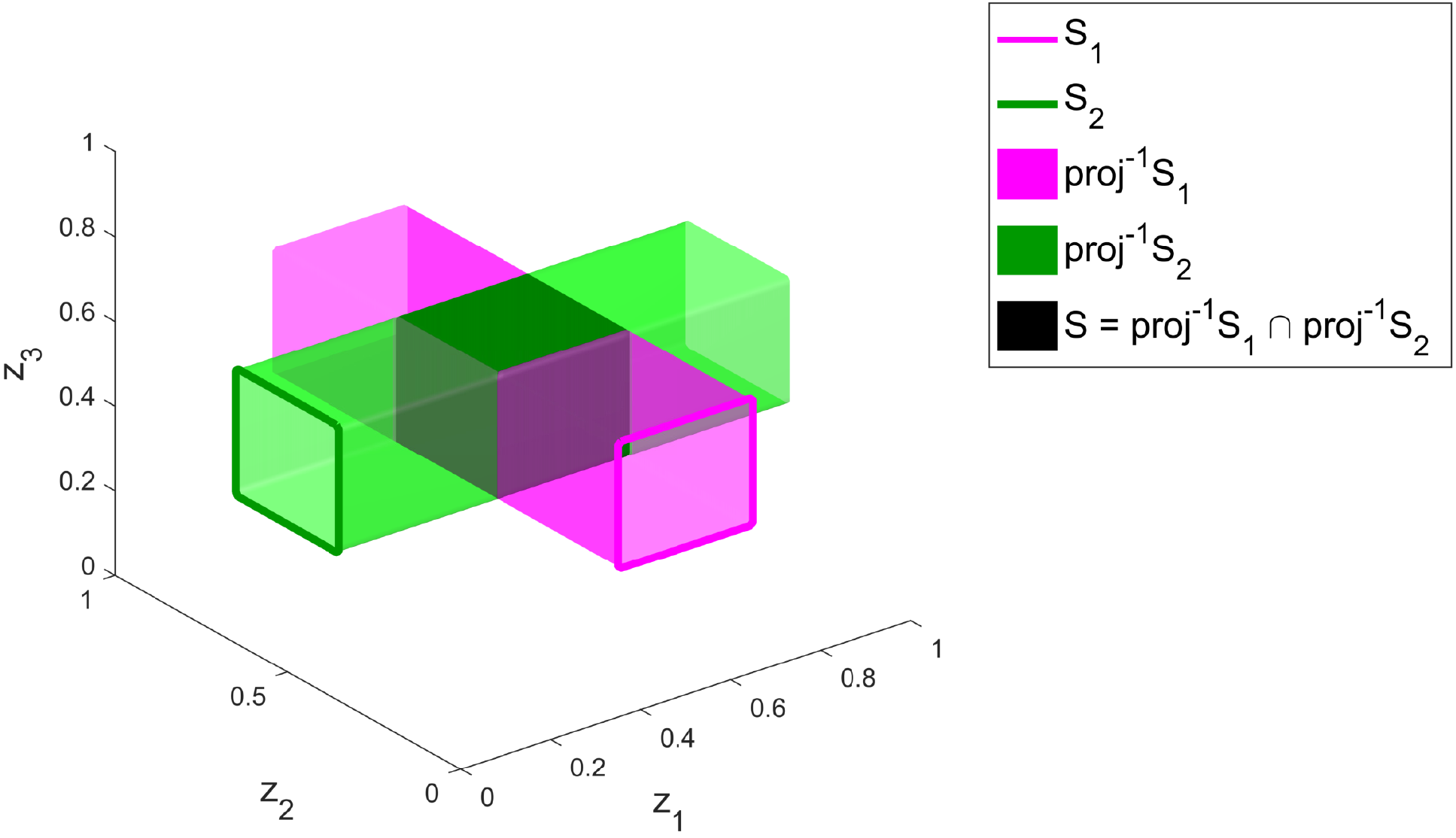}
	\caption{Back projection of sets in the $\state_2$-$\state_3$ plane and the $\state_1$-$\state_3$ plane into the 3D space.}
	\label{fig:proj_set}
\end{figure}

\subsubsection{\textbf{Subsystem Trajectories}}
Since each subsystem in \eqref{eq:scdyn} is self-contained, we can denote the subsystem trajectories $\sctraj_i(s; \scstate_i, t, \ctrl_i(\cdot))$. The subsystem trajectories satisfy the subsystem dynamics and initial condition:

\begin{equation}
\label{eq:scdyn_traj}
\begin{aligned}
\frac{d}{ds}\sctraj(s; \scstate_i, t, \ctrl_i(\cdot)) &= \scdyn_i(\sctraj(s; \scstate_i, t, \ctrl_i(\cdot)), \ctrl_i(s)) \\
\sctraj_i(t; \scstate_i, t, \ctrl_i(\cdot)) &= \scstate_i
\end{aligned}
\end{equation}

The full system trajectory and subsystem trajectories are simply related to each other via the projection operator:
\begin{equation}
\label{eq:proj_traj}
\proj_{\xset_i}(\traj(s; \state, t, \ctrl(\cdot)) = \sctraj_i(s; \scstate_i, t, \ctrl_i(\cdot))
\end{equation}

\noindent where $\scstate_i = \proj_{\xset_i}(\state)$.

\subsection{Goals of This Paper}
We assume that the full system unsafe set $\targetset$ can be written in terms of the subsystem unsafe sets $\targetset_{\scstate_1} \in \xset_1, \targetset_{\scstate_2} \in \xset_2$ in the way depicted in Fig. \ref{fig:proj_set}:

\begin{equation}
\label{eq:target_intersect}
\targetset = \bp(\targetset_{\scstate_1}) \cap \bp(\targetset_{\scstate_2})
\end{equation}

\noindent where the full unsafe set is the intersection of the back projections of subsystem unsafe sets. In practice, this is not a strong assumption since many obstacles can be accurately modeled as rectangular prisms in position space, or hyper-rectangles in the full state space. In fact, the unsafe set described by \eqref{eq:target_intersect} turns out to only be rectangular in the non-overlapping states, and can be arbitrarily shaped in the overlapping states. In addition, such an assumption is reasonable since the full-dimensional unsafe set should at least be representable in some way in the lower-dimensional spaces. However, in the worst case, taking $\targetset_{\scstate_i} = \proj_{\xset_i}(\targetset)$ always leads to a conservative approximation of the constructed BRS that over-approximates the true BRS. Also note that with the definition in \eqref{eq:target_intersect}, we have that $\proj_{\xset_i}(\targetset) = \targetset_{\xset_i}$.

Next, we define the subsystem BRSs $\brs_{\scstate_1}, \brs_{\scstate_2}$ the same way as in \eqref{eq:rset_avoid}, but with the subsystems in \eqref{eq:scdyn} and subsystem unsafe sets $\targetset_{\scstate_1}, \targetset_{\scstate_2}$, respectively:

\begin{equation}
\label{eq:sc_rset}
\brs_{\scstate_i}(t) = \{\scstate_i: \forall \ctrl_i(\cdot) \in \cfset_i, \sctraj_i(0; \scstate_i, t, \ctrl_i(\cdot)) \in \targetset_{\scstate_i} \} \\
\end{equation}

Given a system in the form of \eqref{eq:scdyn} with unsafe set that can be represented by \eqref{eq:target_intersect}, our goal is to compute the full-dimensional BRS by performing computations in the lower-dimensional subspaces. Specifically, we would like to first compute the subsystem BRSs $\brs_{\scstate_1}(t), \brs_{\scstate_2}(t)$, and then construct the full system BRS $\brs(t)$ exactly. This process dramatically reduces computation complexity by decomposing the higher-dimensional system into two lower-dimensional subsystems. Specifically, we will show that if the unsafe set can be decomposed in the way described by \eqref{eq:target_intersect}, then the full-dimensional BRS is decomposable in the same way:

\begin{equation}
\label{eq:recon}
\brs(t) = \bp(\brs_{\scstate_1}(t)) \cap \bp(\brs_{\scstate_2}(t))
\end{equation}

It is important to note that if the subsystem states $\scstate_1, \scstate_2$ have no overlapping states (and are therefore decoupled), the above statement is relatively intuitive and easy to show; however, when the subsystems have the overlapping states in the partition $\spart_3$, they are coupled to each other through these overlapping states. Our main result in this paper proves that despite this coupling, \eqref{eq:recon} still holds.

\section{Self-Contained Subsystems \label{sec:sc}}
With the background and definitions established, we now show the main result in a theorem, which relates lower-dimensional BRSs to the full-dimensional BRS we would like to compute. The consequence of the theorem is that for systems of the form \eqref{eq:scdyn}, one can obtain the \textit{exact} full-dimensional BRS by first computing the lower-dimensional BRSs $\brs_{\xset_i}$, and then constructing the full-dimensional BRS $\brs(t)$ via \eqref{eq:recon}. We first prove a lemma involving a key property of the projection operator.

\begin{lem}
\label{lem:proj_basic}
Let $\bar\state\in\zset, \bar\scstate_i = \proj_{\xset_i}(\bar\state), \genset_i \subseteq \xset_i$ for some subsystem $i$. Then,

\begin{equation}
\bar\scstate_i \in \genset_i \Leftrightarrow \bar\state \in \bp(\genset_i)
\end{equation}
\end{lem}

\begin{proof}
Forward direction: Suppose $\bar\scstate_i \in \genset_i$, then trivially $\exists \scstate_i \in \genset_i, \proj_{\xset_i}(\bar\state) = \scstate_i$ (the $\scstate_i$ that ``exists'' is just $\bar\scstate_i$ itself). By the definition of back projection in \eqref{eq:backproject_set}, we have $\bar\state \in \bp(\genset_i)$.

Backward direction: Suppose $\bar\state \in \bp(\genset_i)$, then by the definition of back projection in \eqref{eq:backproject_set}, we have $\exists \scstate_i \in \genset_i, \proj_{\xset_i}(\bar\state) = \scstate_i$.

Let such an $\scstate_i \in \genset_i$ be denoted $\hat\scstate_i$, and suppose $\bar\scstate_i \notin \genset_i$. Then, we must have $\hat\scstate_i \neq \bar\scstate_i$, which is a contradiction, since $\bar\scstate_i = \proj_{\xset_i}(\bar\state) = \hat\scstate_i$.
\end{proof}

\begin{cor}
\label{cor:proj}
If $\genset = \bp(\genset_1) \cap \bp(\genset_2)$, then 
\begin{equation}
\label{eq:projprop_intersect2}
\bar\state \in \genset \Leftrightarrow \forall i,\bar \scstate_i \in \genset_i\\
\end{equation}
\end{cor}

We now use Lemma \ref{lem:proj_basic} and Corollary \ref{cor:proj} to prove our main result.

\begin{thm}
\label{thm:sc_reach}
\textbf{System decomposition for computing the BRS}. Suppose that the full system in \eqref{eq:fdyn} can be decomposed into the form of \eqref{eq:scdyn}, then

\begin{equation} \label{eq:targetsetconstruction}
\begin{aligned}
&\targetset = \bp(\targetset_{\scstate_1}) \cap \bp(\targetset_{\scstate_2}) \\
&\Rightarrow \brs(t) = \bp(\brs_{\scstate_1}(t)) \cap \bp(\brs_{\scstate_2}(t))
\end{aligned}
\end{equation}
\end{thm}

\begin{IEEEproof}
We will prove Theorem \ref{thm:sc_reach} by proving the following equivalent statement:
\begin{equation}
\label{eq:proof_left}
\bar\state \in \brs(t) \Leftrightarrow \bar\state \in \bp(\brs_{\scstate_1}(t)) \cap \bp(\brs_{\scstate_2}(t))
\end{equation}

By the definition of BRS in \eqref{eq:rset_avoid}, we have
\begin{equation}
\bar\state \in \brs(t) \Leftrightarrow \forall \ctrl(\cdot) \in \cfset, \traj(0; \bar\state, t, \ctrl(\cdot)) \in \targetset
\end{equation}

Consider the property \eqref{eq:proj_traj}, and let 
\begin{equation}
\begin{aligned}
\bar\scstate_i &= \proj_{\xset_i}(\bar\state) \\
 \sctraj_i(0; \bar\scstate_i, t, \ctrl_i(\cdot)) &= \proj_{\xset_i}(\traj(0; \bar\state, t, \ctrl(\cdot)))
\end{aligned}
\end{equation}

Noting that $\targetset = \bp(\targetset_{\scstate_1}) \cap \bp(\targetset_{\scstate_2})$ and using Corollary \ref{cor:proj}, we have the following equivalent statement in terms of the subsystem trajectories:

\begin{equation}
\forall i, \forall \ctrl_i(\cdot), \sctraj_i(0; \bar\scstate_i, t, \ctrl_i(\cdot)) \in \targetset_{\scstate_i}
\end{equation}

\noindent which, by the definition of the subsystem BRS \eqref{eq:sc_rset}, is in turn equivalent to

\begin{equation}
\label{eq:equiv_statement}
\forall i, \bar\scstate_i \in \brs_{\scstate_i}(t)
\end{equation}

By Lemma \ref{lem:proj_basic}, this is equivalent to

\begin{equation}
\forall i, \bar\state \in \bp(\brs_{\scstate_i}(t))
\end{equation}
\end{IEEEproof}

With the above theorem, we now summarize our main theoretical result and its consequences with the following algorithm:
\begin{alg}
\textbf{SCS formulation}. Given an unsafe set $\targetset$ that can be decomposed as $\targetset = \bp(\targetset_{\scstate_1}) \cap \bp(\targetset_{\scstate_2})$ and SCSs with dynamics in the form \eqref{eq:scdyn}, the HJ-based SCS formulation for computing the BRS is given in the following algorithm:
\begin{enumerate}
\item Define the implicit surface functions representing the subsystem unsafe sets $\targetset_{\scstate_1}, \targetset_{\scstate_2}$.
\item Repeat for $i = 1, 2$: For $i$th SCS, compute its BRS by solving \eqref{eq:set_hj} in the space of $\xset_i$.
\item Construct the full-dimensional BRS as follows: $\brs(t) = \bp(\brs_{\scstate_1}(t)) \cap \bp(\brs_{\scstate_2}(t))$. By Theorem \ref{thm:sc_reach}, the full-dimensional BRS is exactly constructed.
\end{enumerate}
\end{alg}

\section{Numerical Examples \label{sec:hd}}
\begin{figure*}[h!]
	\centering
	\includegraphics[width=\textwidth]{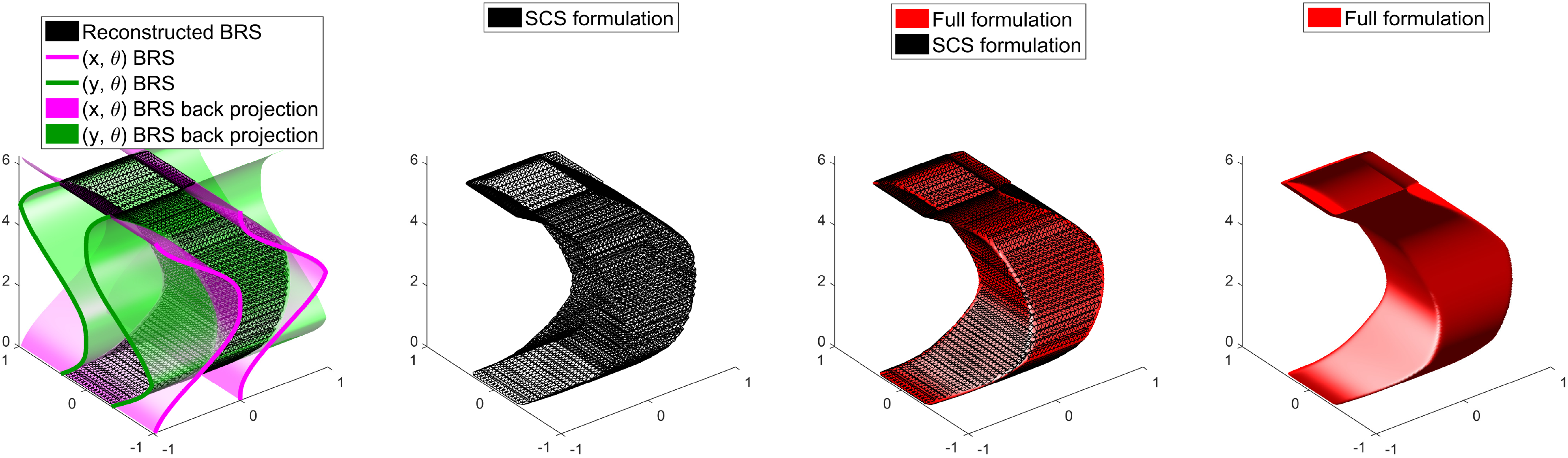}
	\caption{Comparison of the Dubins Car BRS $\brs(t), t=0.5$ computed using the full formulation and the SCS formulation. Left: BRSs in the lower-dimensional subspaces and how they are combined to form the full-dimensional BRS. Middle-left: BRS computed using the SCS formulation. Middle-right: BRSs computed using the full formulation and the BRS formulation superimposed on each other, showing that they are indistinguishable. Right: BRS computed using the full formulation.}
	\label{fig:dubins_compare}
	\vspace{-.2in}
\end{figure*}

We now present two numerical examples to illustrate our method. For each example, we present a common dynamical system that can be decomposed into the form of \eqref{eq:scdyn}. The first example, the 3D Dubins Car, illustrates that our decomposition method produces the exact full-dimensional BRS at a substantially lower computation cost. The second example, the 6D Acrobatic Quadrotor, demonstrates that our technique enables the exact computation of a BRS that was previously intractable to compute with the full formulation.

\subsection{Dubins Car}
The Dubins Car is a well-known system whose dynamics are given by \eqref{eq:dubins_car}. This system is only 3D, and its BRS can be tractably computed in the full-dimensional space, so we use it to compare the full formulation with the SCS formulation. 

\begin{figure}[h]
	\centering
	\includegraphics[width=0.45\textwidth]{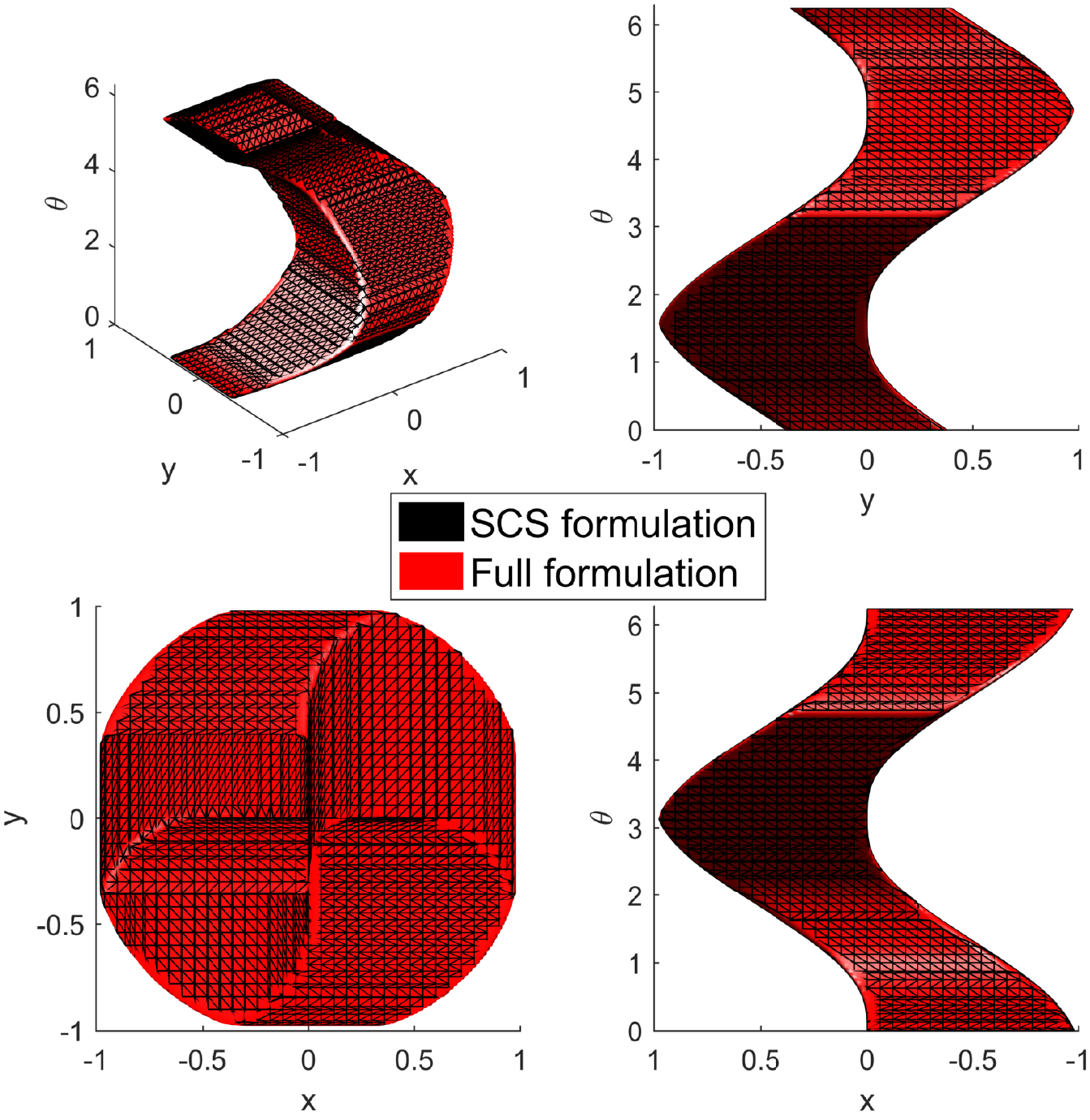}
	\caption{Comparison of the Dubins Car BRS $\brs(t), t=0.5$ computed using the full formulation and the SCS formulation, viewed at a few different angles.}
	\label{fig:dubins_angles}
\end{figure}

\begin{figure}[h]
	\centering
	\includegraphics[width=0.45\textwidth]{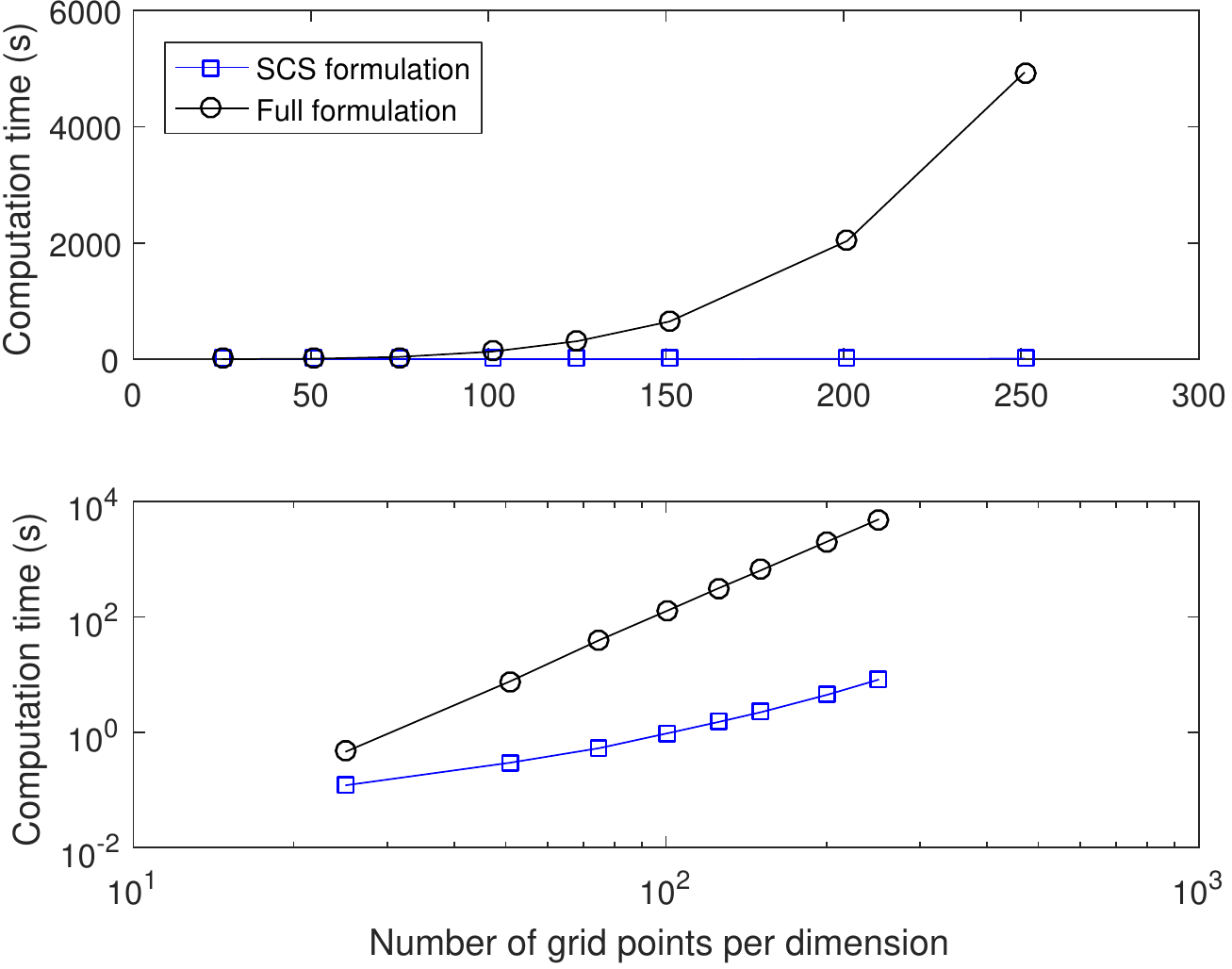}
	\caption{Computation time of the full formulation and the SCS formulation in linear and log scale for the Dubins Car example. In the full formulation, the BRS is computed in a 3D state space, causing the computation time to increase rapidly with the number of grid points per dimension. In contrast, in the SCS formulation, the BRS is computed in a 2D space, and the computation time is negligible in comparison.}
	\label{fig:dubins_time}
\end{figure}

As previously mentioned, the Dubins Car dynamics can be decomposed according to \eqref{eq:dubins_car_decomp}.

For this example, we computed the BRS from the unsafe set representing the set of positions near the origin in both the $\pos_x$ and $\pos_y$ dimensions. More concretely, our unsafe set was defined to be $\targetset = \{(\pos_x, \pos_y, \theta): |\pos_x|, |\pos_y| \le 0.5\}$. Such an unsafe set can be used to model an obstacle that the vehicle must avoid. Given the unsafe set, the interpretation of the BRS $\brs(t)$ is the set of states from which a collision with the obstacle may occur after a duration of $t$.

From $\targetset$, we computed the BRS $\brs(t)$ of time horizon $t=0.5$. The resulting full formulation BRS is shown in Fig. \ref{fig:dubins_compare} as the red surface which appears in the two subplots on the right. 

To compute the BRS using the SCS formulation, note that the unsafe set $\targetset$ can be written as $\targetset = \bp(\targetset_{\scstate_1}) \cap \bp(\targetset_{\scstate_2})$, with

\begin{equation}
\begin{aligned}
\targetset_{\scstate_1} &= \{(\pos_x, \theta): |\pos_x| \le 0.5\} \\
\targetset_{\scstate_2} &= \{(\pos_y, \theta ): |\pos_y| \le 0.5\}
\end{aligned}
\end{equation}

From these lower-dimensional unsafe sets, we computed the lower-dimensional BRSs $\brs_{\xset_1}(t)$ and $\brs_{\xset_2}(t)$, and then constructed the full-dimensional BRS $\brs(t)$ using Theorem \ref{thm:sc_reach}: $\brs(t) = \bp(\brs_{\scstate_1}(t)) \cap \bp(\brs_{\scstate_2}(t))$. The subsystem BRSs and their back projections are shown in magenta and green in the left subplot of Fig. \ref{fig:dubins_compare}. The constructed BRS is shown in the three left subplots of Fig. \ref{fig:dubins_compare} as the black mesh.

In the middle-right plot of Fig. \ref{fig:dubins_compare}, we superimpose the full-dimensional BRS computed using the two methods. We show the comparison of the computation results viewed from several different angles in Fig. \ref{fig:dubins_angles}. The results are indistinguishable. However, when using the SCS formulation, Theorem \ref{thm:sc_reach} allows the computation to be done significantly faster in lower-dimensional subspaces. An additional benefit of the SCS formulation is that in the numerical methods for solving the HJ PDE, the amount of numerical dissipation increases with the number of state dimensions. Thus, computations in lower-dimensional subspaces lead to a slightly more accurate numerical solution.

The computation benefits of using the SCS formulation can be seen from Fig. \ref{fig:dubins_time}. Both subplots show the computation time in seconds versus the number of grid points per dimension in the numerical computation. From the top subplot, one can easily see that the direct computation of the full formulation BRS in 3D becomes very time-consuming as the number of grid points per dimension is increased, while the computation using the SCS formulation hardly takes any time in comparison. The bottom subplot shows the same data, but on a log-log scale for more detail. Directly computing the BRS with 251 grid points per dimension using the full formulation took approximately 80 minutes, while computing the BRS using the SCS formulation is faster by several orders of magnitude: the computation only took approximately 30 seconds! The computations were timed on a desktop computer with an Intel Core i7-2600K processor and 16GB of random-access memory.

\subsection{The 6D Acrobatic Quadrotor}
This example illustrates the ability of the SCS formulation to produce BRSs for high-dimensional systems that would be otherwise intractable to analyze by current HJ-based methods. In \cite{Gillula11}, a 6D Acrobatic Quadrotor model used to perform backflips was simplified into a series of smaller hybrid models due to the intractability of computing a BRS over a 6D state space. Using the new SCS formulation we can accurately compute a BRS for the full 6D system.

The 6D Acrobatic Quadrotor's state is $\state = (\pos_x, \vel_x, \pos_y, \vel_y, \phi, \omega)$; the dynamics are given by \cite{Gillula11}:

\vspace{-0.1in}
\arraycolsep=0pt\def\arraystretch{1.5}
\begin{equation}
\label{eq:quad6D}
\left[
\begin{array}{c}
\dot\pos_x\\
\dot\vel_x\\
\dot\pos_y\\
\dot\vel_y\\
\dot\phi\\
\dot\omega
\end{array}
\right]
=
\left[
\begin{array}{c}
\vel_x\\
\frac{-C^v_D\vel_x}{m}\\
\vel_y\\
\frac{-(mg+C^v_D)\vel_y}{m}\\
\omega\\
\frac{-C^\phi_D\omega}{I_{yy}}
\end{array}
\right]
+
\left[
\begin{array}{cc}
0 & 0\\
\frac{-\sin\phi}{m} & \frac{-\sin\phi}{m}\\
0 & 0\\
\frac{\cos\phi}{m} & \frac{\cos\phi}{m}\\
0 & 0\\
\frac{-l}{I_{yy}} & \frac{l}{I_{yy}}
\end{array}\right]
\left[
\begin{array}{c}
T_1\\
T_2
\end{array}\right]
\end{equation}

\noindent where $\pos_x$, $\pos_y$, and $\phi$ represent the quadrotor's horizontal, vertical, and rotational positions, respectively. Their derivatives represent the velocity with respect to each corresponding positional state. The inputs $T_1$ and $T_2$ represent the thrust exerted on either end of the quadrotor, and the constant system parameters are $m$ for mass, $C_D^v$ for translational drag, $C_D^\phi$ for rotational drag, $g$ for acceleration due to gravity, $l$ for the length from the quadrotor's center to an edge, and $I_{yy}$ for moment of inertia. 

The state partitions of this system are $\spart_1 = (\pos_x, \vel_x), \spart_2 = (\pos_y, \vel_y), \spart_3 = (\phi, \omega)$. Using the SCS formulation, we decompose the full system into the following set of subsystems:
\vspace{-0.1in}
\begin{equation}
\label{eq:quad6Dsc}
\begin{aligned}
&\scstate_1  = \left[
\begin{array}{c}
\spart_1\\
\spart_3
\end{array}
\right] = \left[
\begin{array}{c}
\pos_x\\
\vel_x\\
\phi\\
\omega
\end{array}
\right]
\qquad
\scstate_2 = \left[
\begin{array}{c}
\spart_2\\
\spart_3
\end{array}
\right] = \left[
\begin{array}{c}
\pos_y\\
\vel_y\\
\phi\\
\omega
\end{array}
\right] \\
&\qquad \ctrl_1 = \ctrl_2 = \left[
\begin{array}{c}
T_1 \\
T_2
\end{array}
\right] = \ctrl
\end{aligned}
\end{equation}

For this example we will compute the BRS that describes the set of initial conditions from which the system may enter the unsafe set after a given time period $t$ despite best possible control. We define the unsafe set as a square of length 1m centered at $(\pos_x,\pos_y)=(0,0)$ described by $\targetset = \{(\pos_x,\vel_x, \pos_y, \vel_y, \phi, \omega): |\pos_x|, |\pos_y| \le 1\}$. This can be interpreted as a positional box centered at the origin that must be avoided for all angles and velocities. From the unsafe set, we define $\fc(\state)$ such that $\fc(\state)\le 0 \Leftrightarrow x\in\targetset$. This unsafe set must be decomposed to provide a suitable unsafe set for each subsystem. This is done by letting $\targetset_{\scstate_i},$ $i = 1,2$ be 
\begin{equation}
\begin{aligned}
\targetset_{\scstate_1} &= \{(\pos_x, \vel_x,\phi, \omega): |\pos_x|\le 1\} \\
\targetset_{\scstate_2} &= \{(\pos_y, \vel_y, \phi, \omega): |\pos_y|\le 1\}
\end{aligned}
\end{equation}
The BRS of each 4D subsystem is computed and then combined into the 6D BRS using the SCS formulation. To visually depict the 6D BRS, 3D slices of the BRS along the positional and velocity axes were computed. Fig. \ref{fig:Quad6D_Position} shows a 3D slice in $(\pos_x,\pos_y,\phi)$ space at $\vel_x=\vel_y=1$ m/s, $\omega=0$ rad/s. The dark blue set represents the unsafe set $\targetset$, with the BRS in light blue.
\begin{figure}[H]
	\centering
	\includegraphics[width=.8\linewidth]{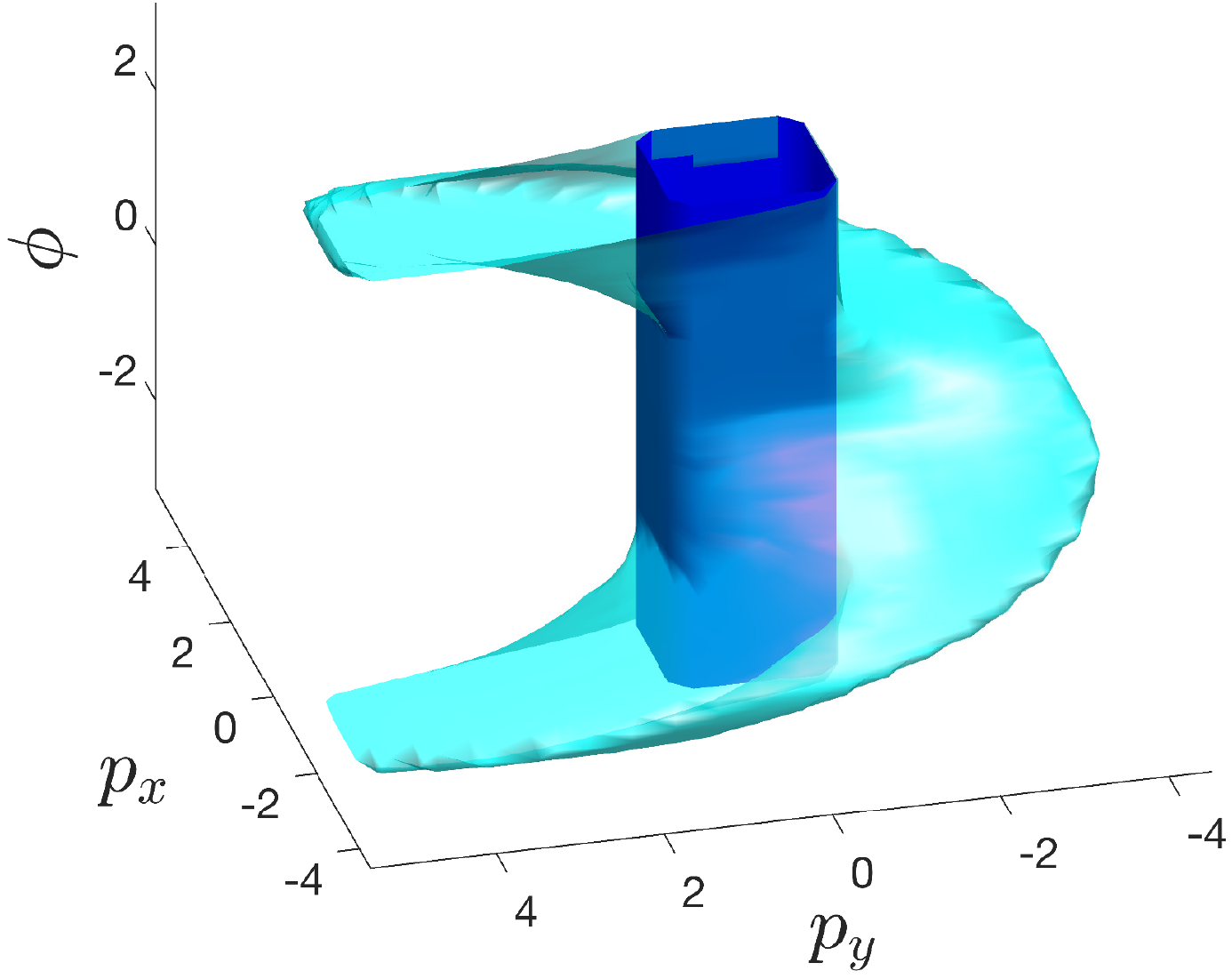}
	\caption{3D slice of the constructed 6D BRS at $\vel_x=\vel_y=1$ m/s, $\omega=0$ rad/s. The unsafe set is in dark blue, with the BRS in light blue.}
	\label{fig:Quad6D_Position}
\end{figure}

In Fig. \ref{fig:Quad6D_Velocity}, 3D slices in $(\vel_x,\vel_y,\omega)$ space are visualized at $\pos_x,\pos_y=1.5$ m, $\phi=1.5$ rad. These colored sets represent the BRS at different points in time. 

\begin{figure}[H]
	\centering
	\includegraphics[width=.7\linewidth]{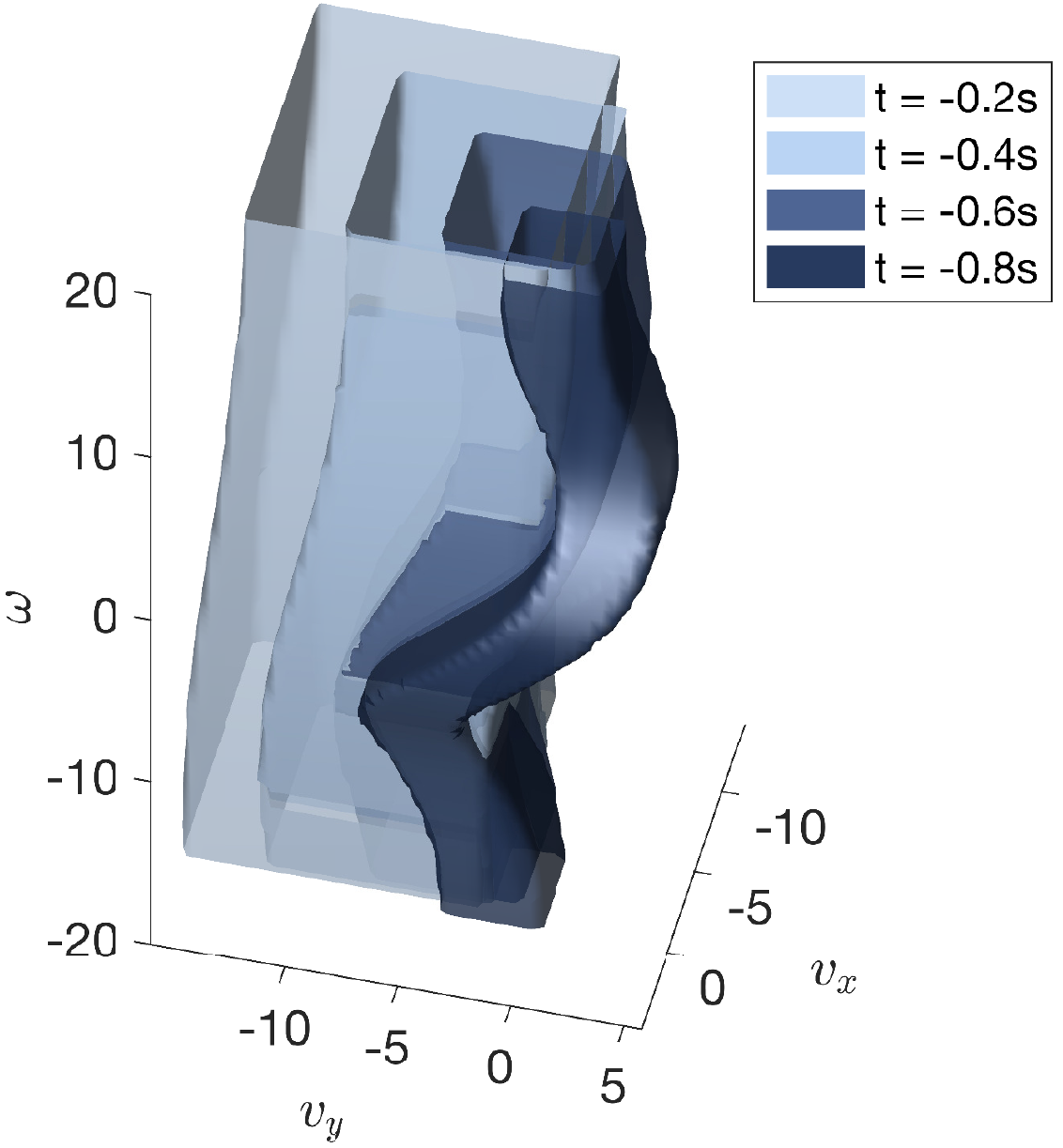}
	\caption{3D slices of the constructed 6D BRS at $\pos_x,\pos_y=1.5$ m, $\phi = 1.5$ rad at different points in time. The sets become darker as $t$ becomes more negative.}
	\label{fig:Quad6D_Velocity}
	\vspace{-.1in}
\end{figure}

\section{Conclusions and Future Work}
The SCS formulation that we proposed for computing BRSs significantly reduces computation burden, and makes many previously intractable computations possible. At the same time, the computation savings do \textit{not} come at the cost of optimality: the full-dimensional BRS can be computed exactly in lower-dimensional subspaces. The construction of the full-dimensional BRS from lower-dimensional BRSs is exact even when the subsystem dynamics are coupled. 

The SCS formulation will be the basis for future system decomposition methods, for which we already have several other preliminary theoretical results. These results include other definitions of BRSs, such as those used for reaching, instead of avoiding, a set; incorporation of disturbances into the problem formulation; and treatment of reachable tubes. In the future, we plan to apply the theory to a larger number of practical systems in these different settings.


\bibliographystyle{IEEEtran}
\bibliography{references}
\end{document}